\documentclass[11pt]{amsart}

\usepackage{amssymb,amsthm,amsmath,amstext}
\usepackage{mathrsfs}  
\usepackage{bm}        
\usepackage{mathtools} 
\usepackage{graphicx}
\usepackage[all]{xy}

\usepackage[left=1.5in,top=1.5in,right=1.5in,bottom=1.55in]{geometry}

\usepackage{amscd}
\usepackage[latin1]{inputenc}
\usepackage[OT1]{fontenc}
\usepackage[bbgreekl]{mathbbol}

\theoremstyle{plain}
\newtheorem{theorem}{Theorem}[section]
\newtheorem{conjecture}[theorem]{Conjecture}
\newtheorem{proposition}[theorem]{Proposition}
\newtheorem{lemma}[theorem]{Lemma}

\newtheorem{notation}[theorem]{Notation}

\theoremstyle{definition}
\newtheorem{definition}[theorem]{Definition}

\newtheorem{remark}[theorem]{Remark}

\newcommand{\sheaf}[1]{\mathscr{#1}}

\newcommand{\PP}{\sheaf{P}}

\newcommand{\Z}{\mathbb Z}

\newcommand{\C}{\mathbb C}
\renewcommand{\P}{\mathbb P}
\newcommand{\Q}{\mathbb Q}

\newcommand{\pullback}{^{*}}

\newcommand{\id}{\mathrm{id}}

\newcommand{\Trans}{{T}}

\usepackage{hyperref}
\hypersetup{pdftitle={The transcendental lattice of the sextic Fermat surface}}
\hypersetup{pdfauthor={Asher Auel, Christian Boehning, and Hans-Christian Graf von Bothmer}}
\hypersetup{colorlinks=true,linkcolor=blue,anchorcolor=blue,citecolor=blue}

\begin{document}
\title[The transcendental lattice of the sextic Fermat surface]{The transcendental lattice of the\\ sextic Fermat surface}

\author[Auel]{Asher Auel$^{1}$}
\address{Asher Auel, Courant Institute of Mathematical Sciences\\
New York University\\
251 Mercer Street\\
New York, NY 10012, USA}
\email{auel@cims.nyu.edu}

\thanks{$^1$ Supported by Mathematical Sciences Postdoctoral Research
Fellowship grant DMS-0903039 of the NSF (United Stated National
Science Foundation).}

\author[B\"ohning]{Christian B\"ohning$^2$}
\address{Christian B\"ohning, Fachbereich Mathematik der Universit\"at Hamburg\\
Bundesstra\ss e 55\\
20146 Hamburg, Germany}
\email{christian.boehning@math.uni-hamburg.de}

\thanks{$^2$ Supported by Heisenberg-Stipendium BO 3699/1-1 of the DFG (German Research Foundation).}

\author[Bothmer]{Hans-Christian Graf von Bothmer$^3$}
\address{Hans-Christian Graf von Bothmer, Fachbereich Mathematik der Universit\"at Hamburg\\
Bundesstra\ss e 55\\
20146 Hamburg, Germany}
\email{hans.christian.v.bothmer@uni-hamburg.de}

\thanks{$^3$ Supported by the RTG 1670 of the DFG (German Research Foundation).}



\begin{abstract} 
We prove that the integral polarized Hodge structure on the
transcendental lattice of a sextic Fermat surface is decomposable.
This disproves a conjecture of Kulikov related to a Hodge theoretic
approach to proving the irrationality of the very general cubic
fourfold.
\end{abstract}

\maketitle

\section{Introduction}
\label{sIntroduction}

Proving the irrationality of a very general cubic fourfold $X \subset
\P^5$ over the complex numbers is a well-known problem in algebraic
geometry.  At present, not a single cubic fourfold is provably
irrational.  However, families of rational cubic fourfolds are
described by Fano \cite{fano}, Tregub \cite{tregub},
\cite{tregub:new}, and Beauville--Donagi \cite{beauville-donagi}.
Hassett \cite{hassett:special-cubics} identifies, via lattice
theory, a countably infinite number of subvarieties, of codimension 2
in the moduli space, consisting of rational cubic fourfolds.  
So far, all known rational cubic fourfolds lie on two divisors of the
moduli space, corresponding to the existence of a plane or a quartic
scroll marking.  Even the construction of additional classes of
rational cubic fourfolds is an open problem.

\medskip

Recently, Kulikov~\cite{kulikov:cubic} initiated a conjectural
approach to the irrationality problem for cubic fourfolds.  The
strategy is modeled on that of Clemens and Griffiths
\cite{clemens_griffiths} for cubic threefolds, with the role of the
intermediate Jacobian played by the integral polarized Hodge structure
$\Trans_X$ on the transcendental part of the middle cohomology
$H^4(X,\Z)$.  Assuming the existence of a birational map $r : \P^4
\dasharrow X$, by Hironaka's resolution of singularities, we can
resolve $r$ to a birational morphism $f : X' \to X$ by a sequence
$$
X' = X_n \to X_{n-1} \to \dotsm \to X_0 = \P^4 \dasharrow X
$$
of blow-ups $X_i \to X_{i-1}$ along points, smooth curves, or smooth
surfaces.  Blow-ups along points, curves, and surfaces of $p_g=0$ do
not contribute to the transcendental lattice, hence there is a
decomposition of polarized Hodge structures
$$
\Trans_{X'} = \textstyle\bigoplus_j \Trans_{S_j}(-1)
$$ 
where the sum is taken over all surfaces $S_j$ of $p_g \geq 1$ that
are the centers of blow-ups in the resolution.  Here, for $S$ a smooth
projective surface, $T_S$ denotes the transcendental part of the
middle cohomology $H^2(S,\Z)$.  On the other hand, there is a
decomposition
$$
\Trans_{X'} = f\pullback \Trans_X \oplus (f\pullback \Trans_X)^{\perp}.
$$
Kulikov proves \cite[Lemmas~2,3]{kulikov:cubic} that for $X$ very
general, comparing these two decompositions yields an index $j_0$ such
that $\Trans_{S_{j_0}} = f\pullback\Trans_{X}(1) \oplus \Trans'$ for
some \emph{nontrivial} polarized Hodge substructure $\Trans' \subset
(f\pullback \Trans_X)^{\perp}$.  The nontriviality of $\Trans'$
follows from standard estimates on the 2nd Betti number of a minimal
model of $S_{j_0}$.  Then the irrationality of the very general cubic
fourfold would follow from the following conjecture.

\begin{conjecture}[Nondecomposability Conjecture {\cite[p.~59]{kulikov:cubic}}]
\label{cKulikovOld} 
Let $S$ be a smooth projective surface over the complex numbers.  Then
the integral polarized Hodge structure $\Trans_S$ on the
transcendental part of $H^2(S,\Z)$ is indecomposable.
\end{conjecture}

For curves, the integral polarized Hodge structure on $H^1$ is indeed indecomposable because of Riemann's theorem describing the theta divisor of a Jacobian in geometric terms and showing that it is irreducible. Hence the Jacobian is indecomposable as a polarized abelian variety and this is what is used substantially in the proof of the irrationality of cubic threefolds by Clemens and Griffiths \cite{clemens_griffiths}. An essential point in Conjecture \ref{cKulikovOld} is the indecomposability over $\mathbb{Z}$. In fact, counterexamples over $\mathbb{Q}$ abound; note also that Jacobians of curves are decomposable if one considers them within the category of abelian varieties up to isogeny.

\medskip

In this paper we prove that Conjecture \ref{cKulikovOld} is false integrally, too:

\begin{theorem}
\label{tMain}
Let $S \subset \P^3$ be the sextic Fermat surface.  Then the integral
polarized Hodge structure $\Trans_S$ on the transcendental part of
$H^2(S,\Z)$ is decomposable.
\end{theorem}
\noindent
After recalling some general theory in Section \ref{sFermatStr}, we prove this in Section \ref{sGaloisSextic}.

\medskip

Admittedly, the sextic Fermat surface $S$, defined in $\P^3$ by
$$
x_0^6+x_1^6+x_2^6 + x_3^6=0
$$
is quite special.  For one, it
has maximal Picard rank $\rho(S) = h^{1,1}(S) = 86$, a fact known to
Beauville, cf.\ \cite[Rem.~3.3(ii)]{shioda:Picard_number} or
\cite[Ex.~3]{schutt:quintic}. The
rank of $\Trans_S$ is 20.

In our analysis, we make use of the description of the integral Hodge
structure of Fermat varieties as a module over the group ring of the
automorphism group in the formulation of
Looijenga~\cite[\S2]{looijenga:Fermat}, which in turn draws on many
previous sources \cite{griffiths:periods}, \cite{shioda:Hodge_Fermat},
\cite{ran:fermat}, \cite{pham:picard}.

\medskip

After a discussion of the results of this paper with the second author
at an Oberwolfach workshop in May 2013, V.~Kulikov suggested that his
conjecture could be modified to the effect that 
surfaces with decomposable integral polarized Hodge structure on the transcendental lattice are rigid with this property, and that this would still imply irrationality of the very general cubic fourfold.
%
%
In Section \ref{sNewIrrationality}, we work out the details of this.  


One may also wonder if surfaces $S$ for which
Conjecture~\ref{cKulikovOld} fails are always defined over
$\overline{\mathbb{Q}}$.

Recall that the sextic Fermat surface has maximal Picard number. It is
not clear to us whether there exists nonisotrivial positive
dimensional families of surfaces with maximal Picard number and $p_g >
1$.

\medskip

Finally, we mention a few other conjectural approaches to prove the
irrationality of the very general cubic fourfold.  A derived
categorical approach due to Kuznetsov \cite{kuznetsov:cubic_fourfolds}
has seen recent activity \cite{addington_thomas}, \cite{BMMS},
\cite{macri_stellari}.  Using the theory of semiorthogonal
decompositions Kuznetsov constructs a triangulated category
$\mathbf{A}_X \subset \mathrm{D}^b(X)$ and conjectures that it encodes
all the information concerning the rationality of $X$.  The
irrationality of the very general cubic fourfold would be a
consequence.  This approach runs into some difficulties due to certain pathologies that semiorthogonal decompositions of derived categories may exhibit, see \cite{bbs:godeaux}, \cite{bbks:barlow}, \cite{bbs:jordan}, but which can possibly be overcome if complemented by new ideas.

There is also a cohomological invariant approach due to
Colliot-Th\'el\`ene.  Many unramified cohomology groups of $X$ vanish
as a consequence of the integral Hodge conjecture
\cite[Thm.~1.1]{colliot_voisin:unramified_cohomology} proved by
Voisin~\cite[Thm.~18]{voisin:aspects} and the triviality of the
Brauer group \cite[Thm.~A.1]{poonen_voloch}.  Nevertheless, by a
result of Merkurjev~\cite[Thm.~2.11]{merkurjev:unramified_cycle_modules}, the vanishing of \emph{all} unramified cohomology
groups arising from cycle modules is controlled by the vanishing of
the Chow group $A_0(X_F)$ of 0-cycles of degree 0 on $X_F = X
\times_{\C} F$ over all field extensions $F/\C$.  Hence the detection
of such a nontrivial 0-cycle on a cubic fourfold over a sufficiently
complicated field $F$ would imply irrationality.

We would like to thank A.~Beauville, F.~Bogomolov, L.~Katzarkov, and
especially V.~Kulikov for discussions and suggestions concerning the
present material.

\section{Integral polarized Hodge structures on Fermat surfaces}
\label{sFermatStr}

Here we describe the integral polarized Hodge structure on the cohomology of Fermat varieties, especially Fermat surfaces, building on and developing further \cite{looijenga:Fermat}. 

\begin{definition}
\label{dPolarizedHodge} \quad \\
\begin{itemize}\setlength{\itemsep}{5pt}
\item[(1)]
An integral polarized Hodge structure (IPHS) of weight $n\in \mathbb{N}$ is a triple $( H_{\Z}, H^{p,q}, Q)$ where
\begin{itemize}\setlength{\itemsep}{5pt}
\item $H_{\Z}$ is a free $\Z$-module of finite rank and $Q : H_{\Z}
\times H_{\Z } \to \Z$ is a nondegenerate (i.e., nondegenerate over
$\Q$) bilinear form with symmetry property $Q(x,y) = (-1)^n
Q(y,x)$ for all $x, y\in H_{\Z}$.
\item
The $H^{p,q}$, $0\le p,q\le n$, are complex linear subspaces in the complexification $H_{\C } = H_{\Z }\otimes_{\Z} \C$ with the property that
\[
H_{\C } = \bigoplus_{p+q=n} H^{p,q}
\]
and such that $H^{p,q} = \overline{H^{q,p}}$, with the conjugation on $H_{\Z} \otimes_{\Z} \C$ being induced by the conjugation on $\C$. 
\item
Extend $Q$ to $H_{\C}$ by linearity. Then we require the orthogonality condition
\[
(x,y) =0 \; \mathrm{if} \; x\in H^{p,q}, \; y\in H^{p', q'} \; \mathrm{with} \; p \neq q' .
\]
Sometimes one also requires the positivity condition
\[
(\sqrt{-1})^{p-q} Q (x, \bar{x}) > 0 \; \mathrm{for} \; 0\neq x \in H^{p,q} .
\]
We choose not to make it part of the abstract notion of integral polarized Hodge structure for definiteness, but this is immaterial for everything that follows: most Hodge structures that occur in this article have this property as they are sub-Hodge structures of geometric Hodge structures on the primitive cohomology of smooth projective varieties.
\end{itemize}
\item[(2)]
The notion of morphism of integral polarized Hodge structures $( H_{\Z , 1}, H_1^{p,q}, Q_1)$ and $( H_{\Z , 2}, H_2^{p,q}, Q_2)$ is the natural one: it is a $\Z$-linear homomorphism $f : H_{\Z , 1} \to H_{\Z , 2}$ which is an isometry, i.e., $Q_2( f(x), f(y)) = Q_1 (x,y)$, and for $f_{\C } = f \otimes \mathrm{id}$ we have $f_{\C } (H^{p,q}_1) \subset H^{p,q}_2$. Such an $f$ is necessarily an embedding since $Q_1$, $Q_2$ are nondegenerate. 
\item[(3)]
There is a natural notion of direct sum of two integral polarized Hodge structures $( H_{\Z , 1}, H_1^{p,q}, Q_1)$ and $( H_{\Z , 2}, H_2^{p,q}, Q_2)$; it is simply given by
\[
(H_{\Z , 1} \oplus H_{\Z , 2}, H^{p,q}_1 \oplus H^{p,q}_2, Q_1 \oplus Q_2) .
\]
An integral polarized Hodge structure is indecomposable if it is not a direct sum of two nontrivial integral polarized Hodge structures.
\end{itemize}
\end{definition}

Let $X$ be an $n$-dimensional smooth projective variety. We assume
that $n=2m$ is even.  Consider the the middle cohomology $H_{\Z } :=
H^n (X, \Z ) /(\mathrm{torsion})$ and its Hodge decomposition $H_{\C }
= H^n (X, C) = \bigoplus_{p+q=n} H^{p,q}$ into
the spaces of harmonic $(p,q)$-forms.  Consider the bilinear form $Q$
defined as the restriction to $H_{\Z}$ of
\[
Q(x,y) = \int_{X} x \wedge y , \qquad\text{for}~ x,y \in H_{\C }.
\]
Then this triple defines an integral polarized Hodge structure; the
positivity condition in Definition \ref{dPolarizedHodge}(1) is not
satisfied, but it holds if we pass to primitive cohomology.

\begin{definition}
\label{dPrimitive}
Let $h \in H^{1,1}\cap H^2 (X, \Z )$ be a polarization class on $X$. 
\begin{enumerate}\setlength{\itemsep}{5pt}
\item
The IPHS on the primitive cohomology of $X$, denoted by  
\[
(H^n_0 (X, \Z ), H^{p,q}_0, Q_0),
\]
is defined as follows: $H^n_0 (X, \Z) \subset H^n (X, \Z
)/(\mathrm{torsion})$ is the sublattice which is orthogonal (with
respect to $Q$) to the middle power of the polarization class $h^m \in H^{m,m}\cap H^n (X, \Z )$ and $H^{p,q}_0 = H^{p,q} \cap (H^n_0 (X, \Z )\otimes_\Z \C )$. Moreover, $Q_0$ is the restriction of $Q$ to $H^n_0 (X, \Z )$. 
\item
We call $A_X = H^{m,m} \cap H^n (X, \Z )$ the algebraic lattice and $T_X = A_X^{\perp}$ the transcendental lattice of $X$. The transcendental IPHS  is 
\[
(T_X, H^{p,q}_T, Q_T)
\]
with  $H^{p,q}_T = H^{p,q} \cap (T_X \otimes \C )$ and $Q_T$ the
restriction of $Q$ to $T_X$.  It is an integral polarized Hodge
substructure of the primitive cohomology.
\end{enumerate}
\end{definition}

\medskip

We will now assume $n=2$ and describe this structure for the Fermat surface $X_d = \{ x_0^d +x_1^d + x_2^d + x_3^d =0 \} \subset \mathbb{P}^3$ of degree $d$ in $\mathbb{P}^3$, taking our point of departure from \cite{looijenga:Fermat}, which we would like to simplify and amplify in several respects. 

Looijenga's computation starts by considering homology. Poincar\'{e} duality gives an isomorphism
\[
P : H_2 (X_d, \Z ) \simeq H^2 (X_d, \Z).
\]

\begin{lemma}
\label{lPD}
If we endow $H_2 (X_d, \Z )$ with the intersection product and $H^2
(X_d, \Z )$ with $Q$, then $P$ is an isomorphism of integral lattices.
\end{lemma}

\begin{proof} 
We have the following commutative diagram:
$$
\xymatrix@R=18pt{
H_2(X_d,\Z) \otimes H_2(X_d,\Z) \ar[r]^-{Q'}\ar[d]^{P\otimes \id} & \Z \ar@{=}[d]\\
H^2(X_d,\Z) \otimes H_2(X_d,\Z) \ar[r]^-{\cap}\ar[d]^{\id\otimes P} & \Z \ar@{=}[d]\\
H^2(X_d,\Z) \otimes H^2(X_d,\Z) \ar[r]^-{Q} & \Z \\
}
$$
where $\cap$ is the topological cap product, $Q'$ the intersection
product, and $Q$ the bilinear form on cohomology as defined above. The
assertion follows.  By abuse of notation, we will also write $Q$ (and not
$Q'$) for the pairing on homology in the following.
\end{proof} 

Looijenga now works with the primitive homology $H_2^0 (X_d, \Z )$
defined as the orthogonal to $h$ (the embedding hyperplane class from
$\P^3$), viewed as an element of $H_2 (X_d, \Z )$ (so this is
$P^{-1}(h)$, to be precise). Hence Poincar\'{e} duality induces an
isomorphism of lattices
\[
P : H_2^0 (X_d, \Z ) \to H^2_0 (X_d, \Z ) .
\]

\begin{remark}
\label{rFermat}
The Fermat hypersurface $X_d$ is invariant under the action of the group $\bbmu^{4}_d/ \bbmu_d$ where $\bbmu^4_d/\bbmu_d$ acts on $X_d$ via rescaling the coordinates. Therefore $H^0_2 (X_d, \Z )$ is naturally a module over the group algebra 
$ \Z [\bbmu^{4}_d/ \bbmu_d].$
\end{remark}

The following is a consequence of \cite{looijenga:Fermat}, Cor.~2.2
and the computation following Rem.~2.3 on p.~6. 

\begin{proposition}
\label{pLattice}
The lattice $H^0_2 (X_d , \Z )$ is  isomorphic, as a $\Z$-module, to the quotient ring
\[
H^0_2 (X_d, \Z ) \simeq \Z [u_0, u_1, u_2, u_3] / I_d
\]
where $I_d$ is the ideal
\[
I_d = \left( u_0  u_1  u_2  u_3 -1, \frac{u_0^d -1}{u_0 -1}, \dots , \frac{u_3^d -1}{u_3 -1} \right) .
\]
The intersection form is given as follows: abbreviating
\[
u^K := u_0^{k_0} \cdot \ldots \cdot u_3^{k_3}\;\mathrm{for} \; K = (k_0, \dots , k_3) \; \mathrm{and} \; \Pi_I := \prod_{i\in I} u_i \; \mathrm{for} \; I \subset \{ 0, \dots , 3 \},
\]
then $u^K \cdot u^L$ is the coefficient of $1$ in 
\[
-u^{K-L} (1 -u_0) (1-u_1)(1-u_2)(1-u_3)
\]
where we calculate in the group ring 
\[
 \Z [\bbmu^{4}_d/ \bbmu_d] = \Z [u_0, u_1, u_2, u_3] /  ( u_0  u_1  u_2  u_3 -1, u_0^d -1, \dots , u_3^d -1 ) .
\]
Moreover, the $\Z[\bbmu^4_d/\bbmu_d]$-module structure on
$H^0_2(X_d,\Z)$ induced by rescaling the coordinates coincides with
its presentation as a submodule of $\Z[\bbmu^4_d/\bbmu_d]$.
\end{proposition}

\begin{remark}
\label{rCharacters}
Let $G= \bbmu^{4}_d/ \bbmu_d$. Fix a primitive $d$-th root of unity $\zeta_d$. The characters $\chi : G \to \C^*$ of $G$ are then given by
\begin{gather*}
\chi (u_i) = \zeta_d^{k_i}, \quad 0 \neq k_i \in \{ 1, \dots , d-1\}
\; \mathrm{and} \; \sum_{0 \leq i \leq 3} k_i \equiv 0 \mod d .
\end{gather*}
Conversely, all tuples $K = (k_0, k_1, k_2, k_3) \in \{ 1, \dots , d-1\}^4$ with zero sum mod $d$ give a character, which we denote by $\chi_K$. Notice that the complex zeros $Z$ of the ideal $I_d$ are precisely the points 
\[
P_K :=(\zeta_d^{k_0}, \dots , \zeta_d^{k_3})
\]
with $K$ as above.
\end{remark}

We now have to describe how $H^0_2 (X_d, \Z ) \simeq H^2_0 (X_d, \Z
)$, viewed as a sublattice of $H^2_0 (X_d, \Z )\otimes \C = H^2_0
(X_d, \C )$, is positioned relative to the Hodge subspaces of $H^2_0
(X_d, \C )$. Note that this will allow us to compute everything: the
algebraic part, the transcendental part, and the induced integral polarized Hodge structure.

\medskip

The Poincar\'{e} duality isomorphism $P$ is equivariant for the natural actions of $G$ on homology and cohomology (it is given by cap product with the fundamental class, which is invariant). Via $P$, we identify $H^{p,q} \subset H^2_0 (X_d, \C )$ with its image in $H^0_2(X_d, \C)$, which we denote by the same letter. 

\begin{proposition}
\label{pHodge}
The Hodge subspace $H^{p,q} \subset L\otimes \C$ is 
\[
H^{p,q} = \bigoplus_{\chi_K } (L\otimes \C )_{\chi_K }
\]
where $(L\otimes \C )_{\chi_K }$ is the eigenspace of the character $\chi_K$ and the sum runs over all characters with 
\[
|K|:= \sum_i k_i = (q+1)d .
\]
In other words, for $Z^{p,q} = \{ P_K \mid \sum_i k_i = (q+1)d \}$, 
\[
H^{p,q} = \left\{  \varphi \in \C [u_0, \dots , u_3]/I_d \mid \mathrm{supp} (\varphi ) \subset Z^{p,q} \right\}
\]
\end{proposition}

\begin{proof}
We can identify the corresponding character spaces in $L\otimes_{\Z } \C \simeq H^0_2 (X_d, \C )$ and $H^2_0 (X_d, \C )$ via the diagram
$$
\xymatrix{
H^0_2(X_d,\C) = L_{\C }  \ar[r]^{P}_{\simeq} & H^2_0 (X_d, \C )\\
H^0_2(X_d,\Z) = L  \ar[r]^{P}_{\simeq} \ar@{^{(}->}[u] & H^2_0 (X_d, \Z )\ar@{^{(}->}[u] 
}
$$
We then apply \cite{looijenga:Fermat}, section after Cor.~2.4 on p.~8.
\end{proof}

\begin{lemma}
\label{lQuadratic}
The quadratic form $Q$ can be written as
\[
Q (\varphi , \psi ) = \sum_{P \in Z } \alpha_P\, \varphi (P) \psi (\bar{P})
\]
for some $\alpha_P \in \Q (\zeta_d)^*$ for each $P \in Z$.
\end{lemma}

\begin{proof}
First, note that $Q$ is invariant under $G$ and also satisfies 
\[
Q(v,w) = Q(gv,gw) = \chi(g)\chi'(g) Q(v,w),
\]
for all $v\in (L\otimes\C)_{\chi}$, $w \in (L\otimes\C)_{\chi'}$ and
all $g\in G$.  Hence $Q(v,w) \neq 0$ only if $\chi =
(\chi')^{-1}=\overline{\chi'}$.
\end{proof}

\begin{remark}
\label{rAlpha}
The whole construction up to now is also invariant under the symmetric group $\mathfrak{S}_4$ acting by permutations on the $u_i$. Therefore $\alpha_P$ is constant on the orbits of the action of $\mathfrak{S}_4$ on $Z$.
\end{remark}

\begin{proposition}
\label{pAlgebraic}
Consider the action of the Galois group $\Gamma = \mathrm{Gal} (\Q (\zeta_d )/\Q )$ on $Z$. Let
\[
Z_A := \left\{ P_K \mid \Gamma \cdot P_K \subset Z^{1,1} \right\} 
\]
and $Z_T := Z \backslash Z_A$.
Then
\[
A_X = \left\{  \varphi \in \Z [u_0, \dots , u_3]/I_d \mid \mathrm{supp} (\varphi ) \subset Z_A \right\}
\]
and 
\[
T_X = \left\{  \varphi \in \Z [u_0, \dots , u_3]/I_d \mid \mathrm{supp} (\varphi ) \subset Z_T \right\}.
\]
\end{proposition}

\begin{proof}
By \cite{shioda:Hodge_Fermat}, Theorem I(iii), we get the assertion about $A_X$. The assertion about $T_X$ then follows from Lemma \ref{lQuadratic}.
\end{proof}

\section{The Fermat sextic}
\label{sGaloisSextic}

\newcommand{\fermat}{X_6}
\newcommand{\Sfour}{\mathfrak{S}_4}

Let now $\fermat$ be the sextic Fermat surface in $\P^3$ and $\zeta$ a
primitive $6$th root of unity. Here $\Gamma = \mathrm{Gal} (\Q
(\zeta_d )/\Q )$ is generated by complex conjugation. Therefore,
\[
A_X\otimes \C = H^{1,1}
\]
and $X_6$ is a surface of maximal Picard rank. Hence 
\[
T_X \otimes \C = H^{2,0} \oplus H^{0,2}.
\]

\begin{notation}
\label{nSubspaces}
Let
\begin{align*}
Z_{(1,1,1,3)} &:= \left\{ P_{(1,1,1,3)} , P_{(1,1,3,1)}, P_{(1,3,1,1)}, P_{(3,1,1,1)}, \right.\\
                        & \left. \quad\quad P_{(5,5,5,3)}, P_{(5,5,3,5)}, P_{(5,3,5,5)}, P_{(3,5,5,5)} \right\}, \\
Z_{(1, 1,2,2)} &:= \left\{ P_{(1,1,2,2)}, P_{(2,2,1,1)}, P_{(5,5,4,4)}, P_{(4,4,5,5)} \right\} , \\
Z_{(1, 2,1,2)} &:= \left\{ P_{(1,2,1,2)}, P_{(2,1,2,1)}, P_{(5,4,5,4)}, P_{(4,5,4,5)} \right\} , \\
Z_{(1, 2,2,1)} &:= \left\{ P_{(1,2,2,1)}, P_{(2,1,1,2)}, P_{(5,4,4,5)}, P_{(4,5,5,4)} \right\} 
\end{align*}
and
\[
L_{\beta } := \left\{  \varphi \in L \mid \mathrm{supp} (\varphi ) \subset Z_\beta \right\}.
\]
\end{notation}
We have
\[
Z_T = Z_{(1,1,1,3)} \cup Z_{(1,1,2,2)} \cup Z_{(1,2,1,2)} \cup Z_{(1,2,2,1)}
\]
and 
\[
T_X \otimes_{\Z} \Q = \left(  L_{(1,1,1,3)} \oplus L_{(1,1,2,2)} \oplus L_{(1,2,1,2)} \oplus L_{(1,2,2,1)} \right)  \otimes_{\Z}\Q .
\]
In the rest of this section we show that this decomposition holds even over $\Z$. The necessary computations were checked using a Macaulay2 script \cite{ABB13}, \cite{M2}. 

\begin{proposition}
\label{pBasis}
There is a sublattice $L'_{(1,1,1,3)}$ of $L_{(1,1,1,3)}$ with a basis such that the intersection form is given by 
\[
Q_{(1,1,1,3)}=
\begin{pmatrix}
       32&
       8&
       8&
       8&
       4&
       16&
       16&
       16\\
       8&
       32&
       8&
       8&
       16&
       4&
       16&
       16\\
       8&
       8&
       32&
       8&
       16&
       16&
       4&
       16\\
       8&
       8&
       8&
       32&
       16&
       16&
       16&
       4\\
       4&
       16&
       16&
       16&
       32&
       8&
       8&
       8\\
       16&
       4&
       16&
       16&
       8&
       32&
       8&
       8\\
       16&
       16&
       4&
       16&
       8&
       8&
       32&
       8\\
       16&
       16&
       16&
       4&
       8&
       8&
       8&
       32\\
       \end{pmatrix}.
\]
We have $\det Q_{(1,1,1,3)} = 2^{16}3^{12}$.
\end{proposition}

\begin{proof}
Consider the matrix
\[
	M_{(1,1,1,3)} = 12 (\zeta+1) \begin{pmatrix}
	\zeta^4 & \zeta^2 & \zeta^2 & \zeta^2 & \zeta^1 & \zeta^3 & \zeta^3 & \zeta^3 \\ 
     	\zeta^2 & \zeta^4 & \zeta^2 & \zeta^2 & \zeta^3 & \zeta^1 & \zeta^3 & \zeta^3 \\ 
     	\zeta^2 & \zeta^2 & \zeta^4 & \zeta^2 & \zeta^3 & \zeta^3 & \zeta^1 & \zeta^3 \\ 
     	\zeta^2 & \zeta^2 & \zeta^2 & \zeta^4 & \zeta^3 & \zeta^3 & \zeta^3 & \zeta^1 \\ 
     	\zeta^1 & \zeta^3 & \zeta^3 & \zeta^3 & \zeta^4 & \zeta^2 & \zeta^2 & \zeta^2 \\ 
     	\zeta^3 & \zeta^1 & \zeta^3 & \zeta^3 & \zeta^2 & \zeta^4 & \zeta^2 & \zeta^2 \\ 
     	\zeta^3 & \zeta^3 & \zeta^1 & \zeta^3 & \zeta^2 & \zeta^2 & \zeta^4 & \zeta^2 \\ 
     	\zeta^3 & \zeta^3 & \zeta^3 & \zeta^1 & \zeta^2 & \zeta^2 & \zeta^2 & \zeta^4
	\end{pmatrix}.
\]
Denote by $P_i$ the $i$-th point of $Z_{(1,1,1,3)}$. By interpolation we find polynomials $\varphi_j$ in $\Z [u_0, \dots , u_3]/I_6$ with $\varphi_j (P_i) = (M_{(1,1,1,3)})_{ij}$ and zero on all other points in $Z$. We can choose
\[
\varphi_1 = ({u}_{3}^{4},
      {u}_{3}^{3},
      {u}_{3}^{2},
      {u}_{3},
      1
    )
      \begin{pmatrix}1\\
      {s}_{1}^{2}-2 {s}_{2}+{s}_{1}+2\\
      {s}_{1}^{2}-3 {s}_{2}+1\\
      -{s}_{1}^{2} {s}_{2}+2 {s}_{2}^{2}-{s}_{3}+2 {s}_{1}^{2}-5 {s}_{2}+{s}_{1}+1\\
      -{s}_{1}^{2} {s}_{3}+3 {s}_{2} {s}_{3}-{s}_{1}^{2} {s}_{2}+3 {s}_{2}^{2}-{s}_{1} {s}_{3}-3 {s}_{3}-2 {s}_{2}-1\\
      \end{pmatrix}
\]
with $s_1=u_0+u_1+u_2$, $s_2=u_0u_1+u_0u_2+u_1u_2$ and $s_3 = u_0u_1u_2$. By applying appropriate permutations of the variables we obtain $\varphi_2, \varphi_3, \varphi_4$. The remaining polynomials are obtained from these by applying the substitution $u_i \mapsto u_i^{-1}=u_i^5$. This induces complex conjugation on the points.

Let $L'_{(1,1,1,3)}$ be the sublattice of $L_{(1,1,1,3)}$ spanned by the $\varphi_j$.
By Lemma \ref{lQuadratic} and Remark \ref{rAlpha} we have that the intersection form on $L'_{(1,1,1,3)}$ is
\[
	Q_{(1,1,1,3)} =\alpha_{P_{(1,1,1,3)}} M_{(1,1,1,3)}^t M_{\overline{(1,1,1,3)}}       
\]
where $M_{\overline{(1,1,1,3)}} $ is obtained from $M_{(1,1,1,3)}$ by interchanging the first four rows with the last four rows (since complex conjugation interchanges the first four points in $Z_{(1,1,1,3)}$ with the last four points). We compute $\alpha_{P_{(1,1,1,3)}}$ by evaluating $Q(\varphi_1, \varphi_1)$ in two different ways: firstly, by using Looijenga's formula in Proposition \ref{pLattice}, and secondly, by Lemma \ref{lQuadratic}. One finds
\[
\alpha_{P_{(1,1,1,3)}}= \frac{1}{108} .
\]
Direct computation gives the above matrix for $Q_{(1,1,1,3)}$ and its determinant.
\end{proof}

Similarly we have the following.

\begin{proposition}
\label{pBasis2}
There is a sublattice $L'_{(1,1,2,2)}$ of $L_{(1,1,2,2)}$ with a basis such that the intersection form is given by
\[
Q_{(1,1,2,2)}=  \begin{pmatrix}24&
      12&
      0&
      0\\
      12&
      24&
      0&
      0\\
      0&
      0&
      24&
      12\\
      0&
      0&
      12&
      24\\
      \end{pmatrix}.
\]
We have $\det Q_{(1,1,2,2)} = 2^83^6$. The same is true for the lattices $L_{(1,2,1,2)}$ and $L_{(1,2,2,1)}$.
	
\end{proposition}

\begin{proof}
Consider
\[
M_{(1,1,2,2)} = 12 (\zeta +1) \begin{pmatrix}
	\zeta^0 & \zeta^5 & \zeta^1 & \zeta^0 \\
	\zeta^2 & \zeta^3 & \zeta^4 & \zeta^5 \\
	\zeta^5 & \zeta^0 & \zeta^4 & \zeta^5 \\
	\zeta^3 & \zeta^2 & \zeta^1 & \zeta^0 \\
\end{pmatrix}.
\]
Denote by $P_i$ the $i$-th point of $Z_{(1,1,2,2)}$. By interpolation we find polynomials $\psi_j$ in $\Z [u_0, \dots , u_3]/I_6$ with $\psi_j (P_i) = (M_{(1,1,2,2)})_{ij}$ and zero on all other points in $Z$. We can choose
\begin{align*}
	\psi_1 &=  q_1 q_2 r_1^2 - q_1 r_1^2 r_2+q_1 r_2^2-q_1^3 r_1+3 q_1 q_2 r_1+q_2 r_1^2
	-q_1^2r_2-q_1 r_1 r_2 \\ 
	& \quad -q_1^3+3 q_1 q_2 
      -2 q_1^2 r_1+3 q_2 r_1-q_1 r_2
      +r_1r_2-q_1^2 \\
      & \quad +2 q_2-q_1 r_1-r_1^2+2
      r_2-2 q_1-2 r_1-2
\end{align*}
with $q_1 = u_0+u_1$, $
      q_2 = u_0 u_1$, $
      r_1 = u_2+u_3$ and $
      r_2 = u_2 u_3
      $.
Replacing $u_i$ by $u_i^5$ gives $\psi_2$ 
and the values of the second column. The third and forth column are realized by $\psi_3=u_1\psi_1$ and $\psi_4=u_1\psi_2$. 
By Lemma \ref{lQuadratic} and Remark \ref{rAlpha} we have that the intersection form on $L'_{(1,1,2,2)}$ is
\[
	Q_{(1,1,2,2)} = \alpha_{P_{(1,1,2,2)}} M_{(1,1,2,2)}^t M_{\overline{(1,1,2,2)}}       
\]
where $M_{\overline{(1,1,2,2)}} $ is obtained from $M_{(1,1,2,2)}$ by interchanging the first two rows with the last two rows. 

We compute $\alpha_{P_{(1,1,2,2)}}$ by evaluating $Q(\psi_1, \psi_1)$ in two different ways: firstly, by using Looijenga's formula in Proposition \ref{pLattice}, and secondly, by Lemma \ref{lQuadratic}. One finds
\[
\alpha_{P_{(1,1,2,2)}}= \frac{1}{72}.
\]
Direct computation gives the above matrix for $Q_{(1,1,2,2)}$ and its determinant. The existence of $L'_{(1,2,1,2)}$ and $L'_{(1,2,2,1)}$ with the analogous bases follows by symmetry.
\end{proof}

\begin{proposition}
\label{pDecomposition}
Let
\[
T'_X = L'_{(1,1,3,3)}\oplus L'_{(1,1,2,2)} \oplus L'_{(1,2,1,2)} \oplus L'_{(1,2,2,1)}.
\]
Then we have an equality of lattices $T'_X=T_X$. 
In particular, always $L'_{\beta} = L_{\beta}$.
\end{proposition}

\begin{proof}
It is clear that $T'_X$ is a sublattice of $T_X$ of finite index. 
Consider the basis of $T'_X$ consisting of the union of the basis vectors of the $L'_{\beta}$ constructed above. One can check  that the reductions of the vectors of this basis modulo $2$ and $3$ are still linearly independent. Since the discriminant of $T'_X$ is only divisible by primes $2$ and $3$, this proves that there is no sublattice of $L$ which contains $T'_X$ as a proper sublattice of finite index. In particular, $T_X = T_X'$.
\end{proof}

This completes the proof of Theorem \ref{tMain}.

\begin{remark}
\label{rDisciminant}
Proposition \ref{pDecomposition} implies also that the discriminant of the transcendental lattice $T_X$ is $2^{40} 3^{30}$ and consequently the discriminant of the Picard lattice is $-2^{40} 3^{30}$ (the sign is negative since the signature of $Q$ on $H^{1,1}$ is $(1, 85)$). We could not find this number in the literature.
\end{remark}

\begin{remark}
\label{rReally}
We found the matrices $M_{(1,1,1,3)}$ and $M_{(1,1,2,2)}$ as follows: using Proposition \ref{pHodge} we find a $\mathbb{Q}$-basis of $L_{(1,1,1,3)}$ and of $L_{(1,1,2,2)} \oplus L_{(1,2,1,2)} \oplus L_{(1,2,2,1)}$. Clearing denominators, we find vectors in the lattice  $L$ that form a basis over $\mathbb{Q}$ of $L_{(1,1,1,3)}$ and of $L_{(1,1,2,2)} \oplus L_{(1,2,1,2)} \oplus L_{(1,2,2,1)}$, respectively. These vectors generate lattices $M$ and $N$, which are not saturated, however. For each prime $p$ dividing the discriminant of $M$, for example, we reduce a set of basis vectors mod $p$ in the ambient $L$, and if the reductions happen to become linearly dependent, we lift the dependency relation to $\mathbb{Z}$ and find a vector divisible by $p$. Continuing in this way we arrive at a saturated sublattice $M'$ spanning the same $\mathbb{Q}$-subspace as $M$.  Using the LLL-algorithm we find vectors in $M'$ with small coefficients. Among these we choose one with small length; evaluating this on $Z_{(1,1,1,3)}$ gives the first column of $M_{(1,1,1,3)}$. The remaining columns are obtained using the $\mathfrak{S}_4$-symmetry and conjugation.
\end{remark}


\section{Rigidity and transcendental lattice decompositions}
\label{sNewIrrationality}

By a family of surfaces $\pi : \mathscr{S} \to B$, we mean a flat
surjective morphism of schemes or analytic spaces, all of whose fibers
are projective surfaces.  All families considered will actually have
smooth fibers, so that $\pi$ is even a smooth map. A family is
isotrivial if the fibers $\mathscr{S}_b$ for $b \in B(\C)$ are all
isomorphic, equivalently, the family is locally trivial for the
\'etale topology, equivalently, the classifying map to the coarse
moduli space maps to a closed point (if it is defined for the
isomorphism class of the fiber).  If $B$ is a scheme or analytic
space, then by a ``very general" point, we mean any point outside of a
countable union of proper analytic subsets.

\begin{conjecture}
\label{cKulikovStrong}
Let $\pi : \mathscr{S} \to B$ be a family of surfaces over
an analytic space $B$ such that a very general fiber $\mathscr{S}_b$
has decomposable integral polarized Hodge structure on the
transcendental lattice.  Then the family $\mathscr{S} \to B$ is
isotrivial.
\end{conjecture}

The main result of this section is the following.

\begin{theorem}
\label{tStrongImpliesIrrational}
Conjecture~\ref{cKulikovStrong} implies that the very general cubic
fourfold $X$ is irrational.
\end{theorem}

Since the integral polarized Hodge structures on the transcendental
parts $T_X$ of the middle cohomology of cubic fourfolds $X$ are
uncountably many, and, moreover, in \cite{kulikov:cubic} it is proved
that, if such $X$ were rational, their $T_X$ must occur as
\emph{proper} summands of $T_S$, for $S$ a surface, it suffices to
prove the following.

\begin{theorem}
\label{tStrongCountable}
Conjecture~\ref{cKulikovStrong} implies that there are only countably many isomorphism classes of surfaces $S$ with decomposable integral polarized Hodge structure on $T_S$.
\end{theorem}

Remark that surfaces with $p_g \le 1$ have indecomposable
$T_S$, so one can restrict to surfaces with $p_g \ge 2$.

\begin{proof}[Proof of Theorem \ref{tStrongCountable}] We divide the proof
into a series of technical steps.

\medskip

\textbf{Step 1}. There are countably many
families $\mathscr{S}_i \to B_i$ of projective surfaces over
irreducible base varieties $B_i$, such that every isomorphism class of a projective
surface is represented by a fiber of some such family. One can take for example the universal families over the Hilbert schemes of two-dimensional subschemes of $\mathbb{P}^n$, $n \in \mathbb{N}$, since there are only countably many Hilbert polynomials. 

Hence it suffices to prove the following.

\begin{lemma}
\label{lCountableFamily}
Let $\pi : \mathscr{S} \to B$ be a family of surfaces
over an analytic space $B$.  Then there are a countable number of
isomorphism classes of fibers $[\mathscr{S}_b]$, for $b\in B(\C)$,
such that the integral polarized Hodge structure on
$T_{\mathscr{S}_b}$ is decomposable.
\end{lemma}

\medskip

\textbf{Step 2}. We first prove a linear algebra result.  

\begin{lemma}
\label{lLinearAlgebra}
Let $D = \mathrm{SO} (2p, q)/ \mathrm{SO} (2p) \times \mathrm{U}(q)$
be the period domain classifying polarized weight two Hodge structures
on $V:=L\otimes \C$, where $L$ is a fixed lattice. Let $\check{D}$ be
the compact dual of $D$ and $\check{D} = \mathrm{IGrass} (p, V)$, so that $D \subset \check{D}$ is a (classically) open subset.

Let $M \subset L$ be a sublattice of $L$. Let $Z \subset \check{D}$ be the subset such that $M$ is contained in the $H^{1,1}$-part defined by the Hodge filtration corresponding to the points in $Z$ (this is a closed algebraic subset). Then for each $z\in Z \cap D$, we get an induced integral polarized Hodge structure on $M^{\perp} \subset L$, by putting $H^{p,q}_{M^{\perp}} : = H^{p,q} \cap (M^{\perp} \otimes \C )$. Let $Z_M\subset Z\cap D$ be the locus of points $z$ where this induced integral polarized Hodge structure is decomposable. Then $Z_M$ is an intersection with $D$ of countably many locally closed algebraic subvarieties of $\check{D}$. 
\end{lemma} 

To prove this, note first that a given lattice, $M^{\perp}$ in our
case, can decompose in at most a countable number of ways into two
nontrivial summands. We fix such a decomposition of $M^{\perp}$. We
then have to prove that the points $z\in \check{D}$ such that the
corresponding Hodge filtration decomposes in a way compatible with the
fixed decomposition of $M^{\perp}$, form a finite union of locally
closed algebraic subsets of $\check{D}$. This follows from the following.

\begin{lemma}
\label{lFlags}
Let $F = \mathrm{Fl} (k_1, \dots , k_r; V)$ be the flag variety of flags $(F_1, \dots , F_r)$ of type $(k_1, \dots , k_r)$ in a complex vector space $V$. Let $W\subset V$ be a fixed subspace with a given direct sum decomposition $W= W_1\oplus W_2$ into subspaces $W_1 \subset W$ and $W_2 \subset W$. Then the subset $Z_{W_1, W_2}$ of flags such that
\[
F_i \cap W = (F_i \cap W_1) \oplus (F_i \cap W_2), \quad \forall i , 
\]
is a finite union of locally closed algebraic subsets. 
\end{lemma} 

The proof of this Lemma is straightforward: it suffices to prove it in case $F = \mathrm{Fl}( k; V)$ is a Grassmannian. The set of $k$-subspaces of $V$ intersecting $W$ in a subspace of fixed dimension $k'$ is a locally closed subset of the Grassmannian. We now fix $k'\le k$ and also positive integers $a$ and $b$ with $a+b =k'$. We have a morphism
\[
\mathrm{Gr} (a, W_1) \times \mathrm{Gr} (b, W_2) \to \mathrm{Gr} (k', W)
\]
(direct sum of subspaces), whose image is a closed subset $G \subset \mathrm{Gr} (k', W)$. Over $\mathrm{Gr} (k', W)$ we have the tautological bundle $\mathscr{E}$ whose fiber over a point is the given subspace of dimension $k'$ of $W$. We consider the relative Grassmannian
\[
\psi : \mathrm{Gr} (k-k', V \otimes \mathscr{O} / \mathscr{E} ) \to \mathrm{Gr} (k', W) .
\]
Now $\mathrm{Gr} (k-k', V \otimes \mathscr{O} / \mathscr{E} )$ is proper and has a natural morphism $f$ to $\mathrm{Gr} (k, V)$. Its image is closed (but consists also of subspaces whose dimension of intersection with $W$ is strictly larger than $k'$ of course). In any case, $f (\psi^{-1} (G))$, intersected with the locus of subspaces $L$ of dimension $k$ in $V$ whose intersection with $W$ has dimension exactly $k'$, is exactly the locus of such subspaces $L$ such that $L\cap W$ decomposes into a direct sum $L\cap W_1$ of dimension $a$ and $L\cap W_2$ of dimension $b$. Since there are only finitely many choices for $a$ and $b$, the result follows.

\medskip

\textbf{Step 3.} 

Now we continue with the proof of Lemma \ref{lCountableFamily}. Look at the period map
\[
p_{\mathscr{S}} : B \to D 
\]
of the family $\pi : \mathscr{S} \to B$. It is holomorphic by
\cite{griffiths:periodsamerican}. (We may assume without loss of
generality that $B$ is so small that there are no monodromy phenomena,
i.e., that the local system $R^2 \pi_{\ast } \Z_{\mathscr{S}}$ is
trivialized; certainly $B$ is covered by countably many such open
subspaces).  Since the Hodge filtration varies holomorphically, the
locus where a given rational cohomology class of type $2p$ remains of
type $(p,p)$ is a complex analytic subspace of $B$ (even more is true,
see \cite{CDK}, but we do not need this). It follows that the Picard
rank and algebraic lattice $A_{\mathscr{S}_b}$ are constant (equal to
$A$) for $b$ outside a countable union of analytic subsets of $B$, say
$\{ B_i \}_{i \in \mathbb{N} }$. Let $U = B - \bigcup_i B_i$.  Look at
$p^{-1} (Z_A ) \subset B$, where $Z_A \subset D$ is the subset of the
respective period domain from Lemma \ref{lLinearAlgebra}. The
restriction of $\mathscr{S}$ to each irreducible component of $p^{-1}
(Z_A)$ which meets $U$ nontrivially fulfills the hypotheses of
Conjecture \ref{cKulikovStrong}. Hence each such component gives only one isomorphism class of surfaces. Thus the
subset of isomorphism classes of fibers $[\mathscr{S}_b]$ such that $b
\in U(\C)$ and the integral polarized Hodge structure on
$T_{\mathscr{S}_b}$ is decomposable, is countable.

\medskip

\textbf{Step 4.}

We repeat the argument of Step 3 for each of the countably many families $\mathscr{S}|_{B_i} \to B_i$. We get new analytic subsets $\{ B_{ij} \}$ in each $B_i$ in this way (countably many) and repeat the argument for the $\mathscr{S}|_{B_{ij}} \to B_{ij}$, and so forth. Each time, the dimension of the base decreases, and after finitely many steps, we reach zero-dimensional bases (countably many). This concludes the proof of Lemma \ref{lCountableFamily}, and with it, the proof of Theorem \ref{tStrongCountable}. 
\end{proof}


\providecommand{\bysame}{\leavevmode\hbox to3em{\hrulefill}\thinspace}

\end{document}